\documentclass[12pt]{article}
\usepackage{amssymb}
\usepackage{amsfonts}
\usepackage{amsmath}
\usepackage{harvard}
\usepackage{graphicx}

\newtheorem{theorem}{Theorem}

\newtheorem{condition}{Condition}

\newtheorem{lemma}{Lemma}

\newenvironment{proof}[1][Proof]{\textbf{#1.} }{\ \rule{0.5em}{0.5em}}
\input{tcilatex}
\renewcommand{\cite}{\citeasnoun}
\renewcommand{\thepage}{}
\renewcommand{\thefootnote}{\fnsymbol{footnote}}
\renewcommand{\baselinestretch}{1.20}

\begin{document}

\title{Refining the Central Limit Theorem Approximation via Extreme Value
Theory}
\author{Ulrich K. M\"{u}ller \\
%EndAName
Economics Department\\
Princeton University}
\date{February 2017}
\maketitle

\begin{abstract}
We suggest approximating the distribution of the sum of independent and
identically distributed random variables with a Pareto-like tail by
combining extreme value approximations for the largest summands with a
normal approximation for the sum of the smaller summands. If the tail is
well approximated by a Pareto density, then this new approximation has
substantially smaller error rates compared to the usual normal approximation
for underlying distributions with finite variance and less than three
moments. It can also provide an accurate approximation for some infinite
variance distributions.

\textbf{AMS 2000 subject classification: }60F05, 60G70, 60E07

\textbf{Key words and phrases: }Regular variation, rates of convergence
\end{abstract}

\newpage \setcounter{page}{1} \renewcommand{\thepage}{\arabic{page}} %
\renewcommand{\thefootnote}{\arabic{footnote}} 
%TCIMACRO{%
%\TeXButton{baselinestretch}{\renewcommand{\baselinestretch}{1.3} \small \normalsize}}%
%BeginExpansion
\renewcommand{\baselinestretch}{1.3} \small \normalsize%
%EndExpansion

\section{Introduction}

Consider approximations to the distribution of the sum $S_{n}=%
\sum_{i=1}^{n}X_{i}$ of independent mean-zero random variables $X_{i}$ with
distribution function $F$. If $\sigma _{0}^{2}=\int x^{2}dF(x)$ exists, then 
$n^{-1/2}S_{n}$ is asymptotically normal by the central limit theorem. The
quality of this approximation is poor if $\max_{i\leq n}|X_{i}|$ is not much
smaller than $n^{1/2}$, since then a single non-normal random variable has
non-negligible influence on $n^{-1/2}S_{n}$. Extreme value theory provides
large sample approximations to the behavior of the largest observations,
suggesting that it may be fruitfully employed in the derivation of better
approximations to the distribution of $S_{n}$.

For simplicity, consider the case where $F$ has a light left tail and a
heavy right tail. Specifically, assume $\int_{-\infty
}^{0}|x|^{3}dF(x)<\infty $ and 
\begin{equation}
\lim_{x\rightarrow \infty }\frac{1-F(x)}{x^{-1/\xi }}=\omega ^{1/\xi }\text{%
, }\omega >0  \label{Par_intro}
\end{equation}%
for $1/3<\xi <1$, so that the right tail of $F$ is approximately Pareto with
shape parameter $1/\xi $ and scale parameter $\omega $. Let $X_{i:n}$ be the
order statistics. For a given sequence $k=k(n)$, $1\leq k<n$, split $S_{n}$
into two pieces%
\begin{equation}
S_{n}=\sum_{i=1}^{n-k}X_{i:n}+\sum_{i=1}^{k}X_{n-i+1:n}.  \label{S_split}
\end{equation}%
Note that conditional on the $n-k$th order statistic $T_{n}=X_{n-k+1:n},$ $%
\sum_{i=1}^{n-k}X_{i:n}$ has the same distribution as $\sum_{i=1}^{n-k}%
\tilde{X}_{i}$, where $\tilde{X}_{i}$ are i.i.d. from the truncated
distribution $\tilde{F}_{T_{n}}(x)$ with $\tilde{F}_{t}(x)=F(x)/F(t)$ for $%
x\leq t$ and $\tilde{F}_{t}(x)=1$ otherwise. Let $\mu (t)$ and $\sigma
^{2}(t)$ be the mean and variance of $\tilde{F}_{t}$. Since $\tilde{F}%
_{T_{n}}$ is less skewed than $F$, one would expect the distributional
approximation (denoted by \textquotedblleft $\overset{a}{\sim }$%
\textquotedblright ) of the central limit theorem, 
\begin{equation}
\sum_{i=1}^{n-k}X_{i:n}|T_{n}\overset{a}{\sim }(n-k)\mu
(T_{n})+(n-k)^{1/2}\sigma (T_{n})Z\text{ \ \ for\ \ }Z\sim \mathcal{N}(0,1)
\label{normal_intro}
\end{equation}%
to be relatively accurate. At the same time, extreme value theory implies
that under (\ref{Par_intro}),%
\begin{equation}
\sum_{i=1}^{k}X_{n-i+1:n}\overset{a}{\sim }n^{\xi }\omega
\sum_{i=1}^{k}\Gamma _{i}^{-\xi }\text{\ for }\Gamma _{i}=\sum_{j=1}^{i}E_{j}%
\text{, }E_{j}\sim \text{i.i.d. exponential.}  \label{EVT_intro}
\end{equation}%
Combining (\ref{normal_intro}) and (\ref{EVT_intro}) suggests%
\begin{equation}
S_{n}\overset{a}{\sim }(n-k)\mu (n^{\xi }\Gamma _{k}^{-\xi })+n^{1/2}\sigma
(n^{\xi }\Gamma _{k}^{-\xi })Z+n^{\xi }\omega \sum_{i=1}^{k}\Gamma
_{i}^{-\xi }  \label{S_approx_intro}
\end{equation}%
with $Z$ independent of $(\Gamma _{i})_{i=1}^{k}$.

If $\xi <1/2$, the approximate Pareto tail (\ref{Par_intro}) and $\mathbb{E}%
[X_{1}]=0$ imply 
\begin{equation*}
\mu (x)\approx -\frac{\omega ^{1/\xi }x^{1-1/\xi }}{(1-\xi )(1-(x/\omega
)^{-1/\xi })}
\end{equation*}%
and $\sigma ^{2}(x)\approx \sigma _{0}^{2}-\omega ^{1/\xi }\frac{1}{1-2\xi }%
x^{2-1/\xi }$ for $x$ large. From $(n-k)/(n-\Gamma _{k})\overset{a}{\sim }1$%
, this further yields%
\begin{equation}
S_{n}\overset{a}{\sim }-n^{\xi }\frac{\omega }{1-\xi }\Gamma _{k}^{1-\xi
}+n^{1/2}\left( \sigma _{0}^{2}-\frac{\omega ^{2}}{1-2\xi }(\Gamma
_{k}/n)^{1-2\xi }\right) _{+}^{1/2}Z+n^{\xi }\omega \sum_{i=1}^{k}\Gamma
_{i}^{-\xi }  \label{S_scndappr_intro}
\end{equation}%
with $(x)_{+}=\max (x,0)$, which depends on $F$ only through the
unconditional variance $\sigma _{0}^{2}$ and the two tail parameters $%
(\omega ,\xi )$. Note that $\mathbb{E}[\Gamma _{i}^{-\xi }]=\Gamma (i-\xi
)/\Gamma (i)$ and $\mathbb{E}[\Gamma _{k}^{1-\xi }]=\Gamma (1+k-\xi )/\Gamma
(k)=(1-\xi )\sum_{i=1}^{k}\Gamma (i-\xi )/\Gamma (i)$, so the right-hand
side of (\ref{S_scndappr_intro}) is the sum of a mean-zero right skewed
random variable, and a (dependent) random-scale mean-zero normal variable.

Theorem 1 below provides an upper bound on the convergence rate of the error
in the approximation (\ref{S_scndappr_intro}). The proof combines the
Berry-Esseen bound for the central limit theorem approximation in (\ref%
{normal_intro}) and the rate result in Corollary 5.5.5 of \cite{Reiss89} for
the extreme value approximation in (\ref{EVT_intro}). If the tail of $F$ is
such that the approximation in (\ref{EVT_intro}) is accurate, then for both
fixed and diverging $k$ the error in (\ref{S_scndappr_intro}) converges to
zero faster than the error in the usual mean-zero normal approximation. The
approximation (\ref{S_scndappr_intro}) thus helps illuminate the nature and
origin of the leading error terms in the first order normal approximation,
as derived in Chapter 2 of \cite{Hall82}, for such $F$. We also provide a
characterization of the bound minimizing choice of $k$.

If $\xi >1/2$, then the distribution of $n^{-\xi }S_{n}$ converges to a
one-sided stable law with index $\xi $. An elegant argument by \cite%
{LePage81}\ shows that this limiting law can be written as $\omega
\sum_{i=1}^{\infty }\Gamma _{i}^{-\xi }$. The approximation (\ref%
{S_approx_intro}) thus remains potentially accurate under $k\rightarrow
\infty $ also for infinite variance distributions. To obtain a further
approximation akin to (\ref{S_scndappr_intro}), note that (\ref{Par_intro})
implies $\sigma ^{2}(\omega x)-\sigma ^{2}(\omega y)\approx (\omega ^{2}/\xi
)\int_{y}^{x}t^{1-1/\xi }dt$ for large $x,y$. Let $u_{n}=(n/k)^{\xi }$. Then%
\begin{equation}
S_{n}\overset{a}{\sim }-n^{\xi }\frac{\omega }{1-\xi }\Gamma _{k}^{1-\xi
}+n^{1/2}\left( \sigma ^{2}(\omega u_{n})+\frac{\omega ^{2}}{\xi }%
\int_{u_{n}}^{(n/\Gamma _{k})^{\xi }}y^{1-1/\xi }dy\right)
_{+}^{1/2}Z+n^{\xi }\omega \sum_{i=1}^{k}\Gamma _{i}^{-\xi }
\label{S_stable_intro}
\end{equation}%
which depends on $F$ only through the tail parameters $(\omega ,\xi )$ and
the sequence of truncated variances $\sigma ^{2}(\omega u_{n})$. The
approximation (\ref{S_stable_intro}) could also be applied to the case $\xi
<1/2$, so that one obtains a unifying approximation for values of $\xi $
both smaller and larger than $1/2.$ Indeed, for $F$ mean-centered Pareto of
index $\xi $, the results below imply that for suitable choice of $%
k\rightarrow \infty $, this approximation has an error that converges to
zero much faster than the error from the first order approximation via the
normal or non-normal stable limit for $\xi $ close to $1/2$. The approach
here thus also sheds light on the nature of the leading error terms of the
non-normal stable limit, such as those derived by \cite{Christoph92}.

For $\xi >1/2,$ the idea of splitting up $S_{n}$ as in (\ref{S_split}) and
to jointly analyze the asymptotic behavior of the pieces is already pursued
in \cite{Csorgo88}. The contribution here is to derive error rates for
resulting approximation to the distribution of the sum, especially for $%
1/3<\xi <1/2$, and to develop the additional approximation of the truncated
mean and variance induced by the approximate Pareto tail.

The next section formalizes these arguments and discusses various forms of
writing the variance term and the approximation for the case where both
tails are heavy. Section 3 contains the proofs.

\section{Assumptions and Main Results}

The following condition imposes the right tail of $F$ to be in the $\delta $%
-neighborhood of the Pareto distribution with index $\xi $, as defined in
Chapter 2 of \cite{Falk04}.

\begin{condition}
For some $x_{0},\delta ,\omega ,L_{F}>0$ and $1/3<\xi <1$, $F(x)$ admits a
density for all $x\geq x_{0}$ of the form%
\begin{equation*}
f(x)=(\omega \xi )^{-1}(x/\omega )^{-1/\xi -1}(1+h(x))
\end{equation*}%
with $|h(x)|\leq L_{F}x^{-\delta /\xi }$ uniformly in $x\geq x_{0}$.
\end{condition}

As discussed in \cite{Falk04}, Condition 1 can be motivated by considering
the remainder in the von Mises condition for extreme value theory. It is
also closely related to the assumption that the tail of $F$ is second order
regularly varying, as studied by \cite{DeHaan96regular} and \cite{deHaan96}.
Many heavy-tailed distributions satisfy Condition 1: for the right tail of a
student-t distribution with $\nu $ degrees of freedom, $\xi =1/\nu $ and $%
\delta =2\xi $, for the tail of a Fr\'{e}chet or generalized extreme value
distribution with parameter $\alpha $, $\xi =1/\alpha $ and $\delta =1,$ and
for an exact Pareto tail, $\delta $ may be chosen arbitrarily large. In
general, shifts of the distribution affect $\delta $; for instance, a
mean-centered Pareto distribution satisfies Condition 1 only for $\delta
\leq \xi $. See Remark 4 below.

We write $C$ for a generic positive constant that does not depend on $k$ or $%
n$, not necessarily the same in each instance it is used.

\begin{theorem}
Under Condition 1,

(a) for $1/3<\xi <1/2$ 
\begin{multline*}
\sup_{s}\left\vert \mathbb{P}(n^{-1/2}S_{n}\leq s)-\mathbb{P}\left( -n^{\xi }%
\frac{\omega }{1-\xi }\Gamma _{k}^{1-\xi }+n^{1/2}\left( \sigma _{0}^{2}-%
\frac{\omega ^{2}}{1-2\xi }(\Gamma _{k}/n)^{1-2\xi }\right)
_{+}^{1/2}Z\right. \right. \\
\left. \left. +n^{\xi }\omega \sum_{i=1}^{k}\Gamma _{i}^{-\xi }\leq
sn^{1/2}\right) \right\vert \leq C\cdot R(k,n,\xi ,\delta )
\end{multline*}

(b) for $1/3<\xi <1$, $u_{n}=(n/k)^{\xi }$ and $a_{n}=(n\log n)^{-1/2}$ for $%
\xi =1/2$ and $a_{n}=n^{-\max (\xi ,1/2)}$ otherwise,%
\begin{multline*}
\sup_{s}\left\vert \mathbb{P}(a_{n}S_{n}\leq s)-\mathbb{P}\left( -n^{\xi }%
\frac{\omega }{1-\xi }\Gamma _{k}^{1-\xi }+n^{1/2}\left( \sigma ^{2}(\omega
u_{n})+\frac{\omega ^{2}}{\xi }\int_{u_{n}}^{(n/\Gamma _{k})^{\xi
}}y^{1-1/\xi }dy\right) _{+}^{1/2}Z\right. \right. \\
\left. \left. +n^{\xi }\omega \sum_{i=1}^{k}\Gamma _{i}^{-\xi }\leq
s/a_{n}\right) \right\vert \leq C\cdot R(k,n,\xi ,\delta )
\end{multline*}

where%
\begin{equation*}
R(k,n,\xi ,\delta )=\left\{ 
\begin{tabular}{ll}
$n^{-1/2}(n/k)^{3\xi -1}+(k/n)^{\delta }k^{1/2}+k/n$ & $\text{for }1/3<\xi
<1/2$ \\ 
$k^{-\xi }+(k/n)^{\delta }k^{1/2}+k/n$ & $\text{for }1/2\leq \xi <1.$%
\end{tabular}%
\right.
\end{equation*}
\end{theorem}

It is straightforward to characterize the rate for $k$ which minimizes the
bound $R(k,n,\xi ,\delta )$. For two positive sequences $a_{n},b_{n}$, write 
$a_{n}\asymp b_{n}$ if $0<\lim \inf a_{n}/b_{n}\leq \lim \sup_{n\rightarrow
\infty }b_{n}/a_{n}<\infty $.

\begin{lemma}
Let $k^{\ast }\asymp n^{\alpha ^{\ast }}$ with%
\begin{equation*}
\alpha ^{\ast }=\left\{ 
\begin{tabular}{ll}
$\left( \min \left( \frac{6\xi -1}{6\xi },\frac{6\xi +2\delta -3}{6\xi
+2\delta -1}\right) \right) _{+}$ & $\text{for }1/3<\xi <1/2$ \\ 
$\min \left( \frac{2\delta }{1+2(\delta +\xi )},\frac{1}{1+\xi }\right) $ & $%
\text{for }1/2\leq \xi <1.$%
\end{tabular}%
\right.
\end{equation*}%
Then $\min_{k\geq 1}R(k,n,\xi ,\delta )\asymp R(k^{\ast },n,\xi ,\delta
)\asymp n^{\beta ^{\ast }}$ with 
\begin{equation*}
\beta ^{\ast }=\left\{ 
\begin{tabular}{ll}
$-\delta $ & $\text{for }\delta \leq 3(1/2-\xi )$ \\ 
$-\frac{3+2\delta -6\xi }{12\xi +4\delta -2}$ & $\text{for }3(1/2-\xi
)<\delta \leq 1/2+3\xi $ \\ 
$-\frac{1}{6\xi }$ & $\text{for }1/2+3\xi <\delta $%
\end{tabular}%
\right.
\end{equation*}%
for $1/3<\xi <1/2$, and $\beta ^{\ast }=-\xi \alpha ^{\ast }$ for $1/2\leq
\xi <1$.
\end{lemma}

%TCIMACRO{%
%\TeXButton{B}{\begin{figure}[tbp] 
%\begin{center}}}%
%BeginExpansion
\begin{figure}[tbp] 
\begin{center}%
%EndExpansion
\caption{Error Convergence Rates $n^\beta$ of Refined Approximation}

\bigskip

\includegraphics[scale=0.70]{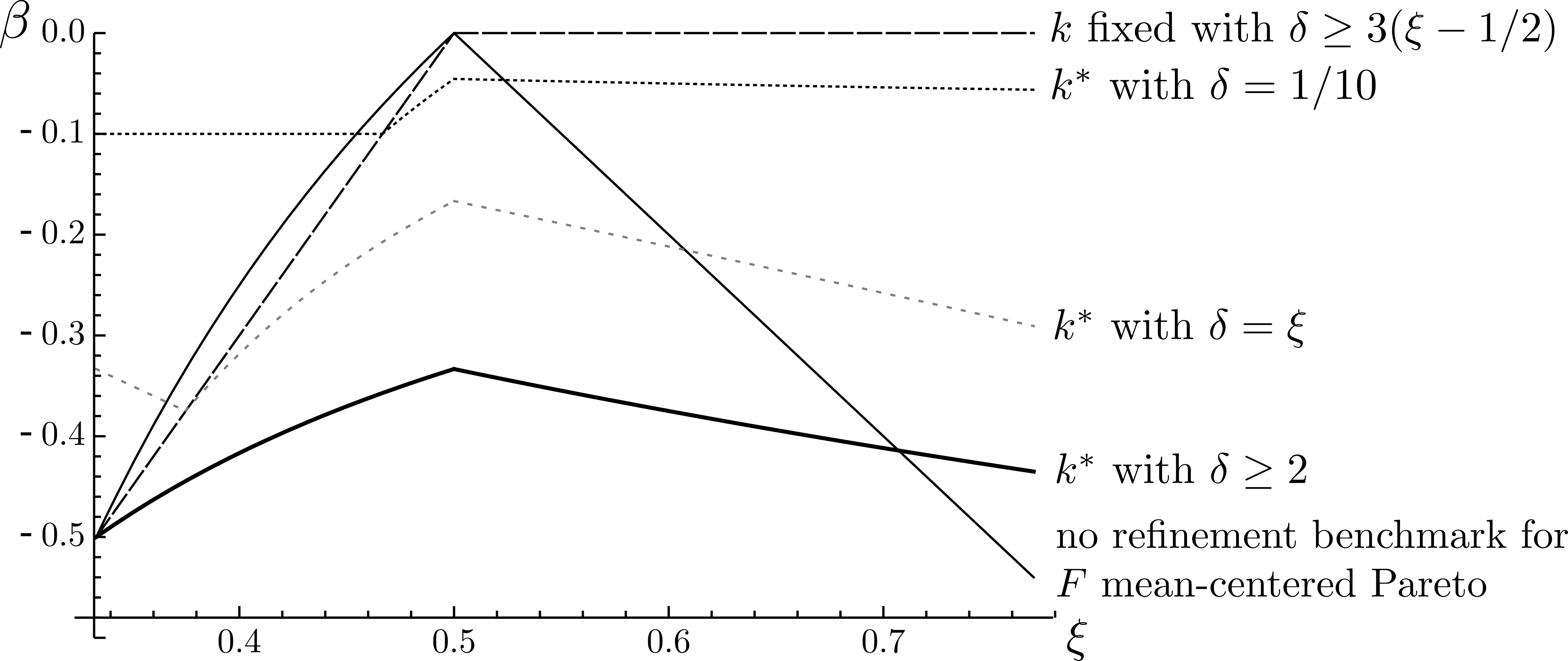}%
%TCIMACRO{%
%\TeXButton{small}{\end{center}
%
%\begin{small}}}%
%BeginExpansion
\end{center}

\begin{small}%
%EndExpansion
%TCIMACRO{\TeXButton{endsmall}{\end{small}}}%
%BeginExpansion
\end{small}%
%EndExpansion
%TCIMACRO{\TeXButton{E}{\end{figure}}}%
%BeginExpansion
\end{figure}%
%EndExpansion
\emph{Remarks.}

1. For $1/3<\xi <1/2$, \cite{Hall79} shows that under Condition 1, the error
in the usual normal approximation to the distribution of $S_{n}$ satisfies $%
\sup_{s}|\mathbb{P}(n^{-1/2}S_{n}\leq s)-\mathbb{P}(\sigma _{0}Z\leq
s)|\asymp n^{1-1/(2\xi )}$, so convergence is very slow for $\xi $ close to $%
1/2$. For $\xi =1/2$, Theorems 3 and 4 in \cite{Hall80b} imply that under
Condition 1, $(n\log n)^{-1/2}S_{n}$ converges to a normal distribution at a
logarithmic rate. For any $\delta >0$, the new approximation with optimal
choice of $k^{\ast }$ yields a better rate $n^{\beta ^{\ast }}$ for $\xi $
sufficiently close to $1/2$, and for sufficiently large $\delta $, the rate
is at least as fast as $n^{-1/3}$ for all $1/3<\xi \leq 1/2$. Thus, if the
tail of $F$ is sufficiently close to being Pareto in the sense of Condition
1, then the new approximations can provide dramatic improvements over the
normal approximation. Even keeping $k$ fixed improves over the benchmark
rate $n^{1-1/(2\xi )}$ as long as $\delta >1/(2\xi )-1$ for $1/3<\xi <1/2$.
At the same time, if $\delta <1/2$, then $\beta ^{\ast }$ is larger than $%
1-1/(2\xi )$ for some $\xi $ sufficiently close to $1/3$, so the new
approximation is potentially worse than the usual normal approximation (or,
equivalently, the optimal choice of $k$ then is $k^{\ast }=0$).

For $1/2<\xi <1$ and under Condition 1, $\sup_{s}|\mathbb{P}(n^{-\xi
}S_{n}\leq s)-\mathbb{P}(\omega \sum_{i=1}^{\infty }\Gamma _{i}^{-\xi }\leq
s)|=O(n^{1-2\xi }+n^{-\delta })$ by Theorem 1 of \cite{Hall81}, and his
Theorem 2 shows this rate to be sharp under a suitably strengthened version
of Condition 1. More specifically, for $F$ mean-centered Pareto, the rate is
exactly $n^{1-2\xi }$ (cf. \cite{Christoph92}, Example 4.25), which, for any 
$\delta >0$, is slower than $n^{\beta ^{\ast }}$ for $\xi $ sufficiently
close to $1/2$.

Figure 1 plots some of these rates.

2. An alternative approximation is obtained by replacing the term in the
positive part function in parts (a) and (b) of Theorem 1 by $\sigma
^{2}(\omega (n/\Gamma _{k})^{\xi })$, with an approximation error that is
still bounded by $C\cdot R(k,n,\xi ,\delta ).$ Substitution of the term $%
\sigma _{0}^{2}-\frac{\omega ^{2}}{1-2\xi }(\Gamma _{k}/n)^{1-2\xi }$ in
part (a) of Theorem 1 by $\sigma _{0}^{2}-\frac{\omega ^{2}}{1-2\xi }%
(k/n)^{1-2\xi }$ (or dropping the integral in part (b) for $1/3<\xi <1/2$)
induces an additional error of order $(k/n)^{1-2\xi }k^{-1/2}$. In general,
this worsens the bound, although even with this further approximation, the
rate can still be better than the baseline rate of $n^{1-1/(2\xi )}$. For $%
1/2<\xi <1$, dropping the integral in part (b) induces an additional error
of order $k^{-\xi }$, so this simpler approximation still has an error no
larger than $C\cdot R(k,n,\xi ,\delta )$.

3. Consider the case where both tails of $F$ are approximately Pareto, that
is Condition 1 holds for $\xi =\xi _{R}$ and $\delta =\delta _{R}$, and for
some $x_{L},\omega _{L},\delta _{L},L_{L}>0,$ for all $x<-x_{L}$, $%
f(x)=(\omega _{L}\xi _{L})^{-1}(-x/\omega _{L})^{-1/\xi _{L}-1}(1+h_{L}(-x))$
with $|h_{L}(x)|\leq L_{L}x^{-\delta _{L}/\xi _{L}}$ for all $x>x_{L}$.
Proceeding as in the introduction then suggests 
\begin{multline*}
S_{n}\overset{a}{\sim }n^{\xi _{L}}\frac{\omega _{L}}{1-\xi _{L}}\Upsilon
_{k_{L}}^{1-\xi _{L}}-n^{\xi _{R}}\frac{\omega _{R}}{1-\xi _{R}}\Gamma
_{k_{R}}^{1-\xi _{R}}+n^{1/2}\sigma (\omega _{L}(n/\Upsilon _{k_{L}})^{\xi
_{L}},\omega _{R}(n/\Gamma _{k_{R}})^{\xi _{R}})Z \\
+n^{\xi }\omega _{R}\sum_{i=1}^{k_{R}}\Gamma _{i}^{-\xi _{R}}-n^{\xi
_{L}}\omega _{L}\sum_{i=1}^{k_{L}}\Upsilon _{i}^{-\xi _{L}}
\end{multline*}%
with $(\Upsilon _{i})_{i=1}^{\infty }$ an independent copy of $(\Gamma
_{i})_{i=1}^{\infty }$ and $\sigma ^{2}(x,y)$ the variance of $X_{1}$
conditional on $-x\leq X_{1}\leq y$. If $1/3<\xi _{L},\xi _{R}<1$, then
arguments analogous to the proof of Theorem 1 show that the error of this
approximation is bounded by an expression of the form $C\cdot R(k_{R},n,\xi
_{R},\delta _{R})+C\cdot R(k_{L},n,\xi _{L},\delta _{L})$, and the same form
is obtained by replacing $\sigma ^{2}(x,y)$ with $\sigma ^{2}(\omega
_{L}x,\omega _{R}y)=(\sigma ^{2}(\omega _{L}v_{n},\omega _{R}u_{n})+(\omega
_{L}^{2}/\xi _{L})\int_{v_{n}}^{x}t^{1-1/\xi _{L}}dt+(\omega _{R}^{2}/\xi
_{R})\int_{u_{n}}^{y}t^{1-1/\xi _{R}}dt)_{+}$ for $v_{n}=(n/k_{L})^{\xi
_{L}} $ and $u_{n}=(n/k_{R})^{\xi _{R}}$ (and the integrals may be dropped
for $1/2<\xi <1$, see the preceding remark). If $\bar{\xi}=\max (\xi
_{L},\xi _{R})>1/2$ and $\xi _{L}\neq \xi _{R}$, then the first order
approximation to the distribution of $n^{-\bar{\xi}}S_{n}$ is a one-sided
stable law that does not depend on the smaller tail index. In contrast, the
approximation above reflects the impact of both heavy tails, and in general,
ignoring the relatively lighter tail leads to a worse bound.

4. Suppose the right tail of $F$ is well approximated by a shifted Pareto
distribution, that is for some $\kappa \in 
%TCIMACRO{\U{211d} }%
%BeginExpansion
\mathbb{R}
%EndExpansion
$ and $x_{1},\delta _{1},L_{1}>0$, $dF(x)/dx=f(x)=(\omega \xi
)^{-1}((x-\kappa )/\omega )^{-1/\xi -1}(1+h(x-\kappa ))$ for all $%
x>x_{1}+\kappa $ with $|h(y)|\leq L_{1}y^{-\delta _{1}/\xi }$ uniformly in $%
y\geq x_{1}$. This implies that $F$ satisfies Condition 1, but only for $%
\delta =\min (\xi ,\delta _{1})$. Let $F_{0}(x)=F(x+\kappa )$ and $\mu
_{0}(x)=-\int_{x}^{\infty }ydF_{0}(y)/F_{0}(x)$. Then $\mu (x+\kappa
)=-\int_{x+\kappa }^{\infty }ydF(y)/F(x+\kappa )=[\mu _{0}(x)-\kappa
(1-F_{0}(x))]/F_{0}(x)$. Thus, proceeding as for (\ref{S_scndappr_intro})
yields $(X_{n-i+1:n})_{i=1}^{k}\overset{a}{\sim }(\kappa +n^{\xi }\omega
\Gamma _{i}^{-\xi })_{i=1}^{k}$ and%
\begin{equation}
S_{n}\overset{a}{\sim }\kappa (k-\Gamma _{k})-n^{\xi }\frac{\omega }{1-\xi }%
\Gamma _{k}^{1-\xi }+n^{1/2}\sigma (\omega (n/\Gamma _{k})^{\xi }+\kappa
)Z+n^{\xi }\omega \sum_{i=1}^{k}\Gamma _{i}^{-\xi }.  \label{S_thrdappr}
\end{equation}%
Straightforward modifications of the proof of Theorem 1 show that the
approximation error in (\ref{S_thrdappr}) is bounded by $C\cdot R(k,n,\xi
,\delta _{1}),$ and this form for the bound also applies if $\sigma
^{2}(\omega x+\kappa )$ is further approximated by $\sigma ^{2}(\omega
x+\kappa )\approx (\sigma ^{2}(\omega u_{n})+(\omega ^{2}/\xi
)\int_{u_{n}}^{x}y^{1-1/\xi }dy)_{+}$ for $u_{n}=(n/k)^{\xi }$. So, for
instance, if $F$ is mean-centered Pareto with $1/3<\xi <1$, then $\delta
_{1} $ may be chosen arbitrarily large, and the approximation (\ref%
{S_thrdappr}) with $k=k^{\ast }$ of Lemma 1 yields a substantially better
bound on the convergence rate compared to the original approximation (\ref%
{S_stable_intro}) with a bound of the form $C\cdot R(k,n,\xi ,\xi )$. The
cost of this further refinement, however, is the introduction of a tail
location parameter $\kappa $ in addition to the tail scale and tail shape
parameters $(\omega ,\xi )$.

\section{Proofs}

Let $X_{n}^{e}=(X_{n-k+1:n},X_{n-k:n},\ldots ,X_{n:n}).$ The proof of
Theorem 1 relies heavily on Corollary 5.5.5 of \cite{Reiss89} (also see
Theorem 2.2.4 of \cite{Falk04}), which implies that under Condition 1, 
\begin{equation}
\sup_{B^{k}}|\mathbb{P}(n^{-\xi }\omega ^{-1}X_{n}^{e}\in B^{k})-\mathbb{P}%
((\Gamma _{k}^{-\xi },\Gamma _{k-1}^{-\xi },\ldots ,\Gamma _{1}^{-\xi })\in
B^{k})|\leq C((k/n)^{\delta }k^{1/2}+k/n)  \label{Reiss555}
\end{equation}%
where the supremum is over Borel sets $B^{k}$ in $%
%TCIMACRO{\U{211d} }%
%BeginExpansion
\mathbb{R}
%EndExpansion
^{k}$.

Without loss of generality, assume $1-(x_{0}/\omega )^{-1/\xi }>0$, $\sigma
_{0}^{2}-\omega ^{1/\xi }x_{0}^{1-2\xi }/(1-2\xi )>0$ and $x_{0}>e$. We
first prove two elementary lemmas. Let $L$ denote a generic positive
constant that does not depend on $x$ or $y$, not necessarily the same in
each instant it is used.

\begin{lemma}
\label{lm:con1}Under Condition 1, for all $x,y\geq x_{0}$,

(a) for $1/3<\xi <1$, $|\mu (x)|\leq Lx^{1-1/\xi }$ and $\left\vert \mu (x)+%
\frac{\omega ^{1/\xi }}{1-\xi }\frac{x^{1-1/\xi }}{1-(x/\omega )^{-1/\xi }}%
\right\vert \leq Lx^{1-(1+\delta )/\xi }$

(b) for $1/3<\xi <1/2$, $|\sigma ^{2}(x)-\sigma _{0}^{2}+\frac{\omega
^{1/\xi }}{1-2\xi }x^{2-1/\xi }|\leq L(x^{2-(1+\delta )/\xi }+x^{2-2/\xi
}+x^{-1/\xi })$

(c) for $1/2<\xi <1$, $L^{-1}x^{2-1/\xi }\leq \sigma ^{2}(x)\leq Lx^{2-1/\xi
}$ and $|\sigma ^{2}(x)-\sigma ^{2}(y)+(\omega ^{1/\xi }/\xi
)\int_{x}^{y}u^{1-1/\xi }du|\leq L(|\int_{x}^{y}u^{1-(1+\delta )/\xi
}du|+(x^{2-1/\xi }+y^{2-1/\xi })(x^{-1/\xi }+y^{-1/\xi }))$

(d) for $\xi =1/2$, $L^{-1}\log x\leq \sigma ^{2}(x)\leq L\log x$, and $%
|\sigma ^{2}(x)-\sigma ^{2}(y)+(\omega ^{1/\xi }/\xi )\int_{x}^{y}u^{1-1/\xi
}du|\leq L(|\int_{x}^{y}u^{1-(1+\delta )/\xi }du|+x^{-1/\xi }\log
y+y^{-1/\xi }\log x)$

(e) for $1/3<\xi <1$, $\int |y|^{3}d\tilde{F}_{x}(y)\leq Lx^{3-1/\xi }$.
\end{lemma}

\begin{proof}
(a) Follows from $\mu (x)=-\int_{x}^{\infty }udF(u)/F(x)$ and, under
Condition 1, $|\int_{x}^{\infty }udF(u)-\frac{\omega ^{1/\xi }}{1-\xi }%
x^{1-1/\xi }|\leq \int_{x}^{\infty }u|h(u)|du\leq Lx^{1-(1+\delta )/\xi }$
and $|F(x)-1+(x/\omega )^{-1/\xi }|\leq \int_{x}^{\infty }|h(u)|du\leq
Lx^{-(1+\delta )/\xi }$.

(b),(c),(d) Since $\sigma ^{2}(x)=-\mu (x)^{2}+(\sigma
_{0}^{2}-\int_{x}^{\infty }u^{2}dF(u))/F(x)$ for $\xi <1/2$ and $\sigma
^{2}(x)=-\mu (x)^{2}+\int_{-\infty }^{x}u^{2}dF(u)/F(x)$ for $\xi \geq 1/2$,
the results follow from $|1-F(x)|\leq Lx^{-1/\xi }$, $%
|\int_{x}^{y}u^{2}dF(u)-(\omega ^{1/\xi }/\xi )\int_{x}^{y}u^{1-1/\xi
}du|\leq L\int_{x}^{y}u^{1-(1+\delta )/\xi }du$ via Condition 1 and the
result in part (a).

(e) Follows from $\int |u|^{3}d\tilde{F}_{x}(u)=\int_{-\infty }^{x}|u-\mu
(x)|^{3}dF(u)/F(x)\leq L\int_{-\infty }^{x}|u|^{3}dF(u)+L|\mu (x)|^{3}$ by
the $c_{r}$ inequality and Condition 1.
\end{proof}

\begin{lemma}
\label{lm:expecT}Under Condition 1

(a) with $\tilde{T}_{n}=\max (T_{n},x_{0})$, $\mathbb{E}[\tilde{T}%
_{n}^{\alpha }]\leq C(n/k)^{\alpha \xi }$ for all $0\leq \alpha <1/\xi $

(b) with $\tilde{\tau}_{n}=\max (\omega n^{\xi }\Gamma _{k}^{-\xi },x_{0})$, 
$\mathbb{E}[\tilde{\tau}_{n}^{-\alpha }]\leq C(n/k)^{-\alpha \xi }$ for all $%
\alpha \geq 0$.
\end{lemma}

\begin{proof}
(a) Let $Y_{n}=(k/n)^{\xi }\tilde{T}_{n}$, so that we need to show that $%
\mathbb{E}[Y_{n}^{\alpha }]$ is uniformly bounded or, equivalently, that $%
\mathbb{P}(Y_{n}\geq y)y^{\alpha -1}$ is uniformly integrable. We have, for $%
y>x_{0}$%
\begin{eqnarray*}
\mathbb{P}(Y_{n}\left. \geq \right. x) &=&\mathbb{P}(T_{n}\geq (n/k)^{\xi }y)
\\
&=&\mathbb{P}(1-F(T_{n})\leq 1-F((n/k)^{\xi }y)) \\
&\leq &\mathbb{P}(U_{k:n}\leq \bar{L}y^{-1/\xi }k/n)
\end{eqnarray*}%
where $U_{k:n}$ is the $k$th order statistic of $n$ i.i.d. uniform $[0,1]$
variables, and $\bar{L}$ is such that $1-F(x)\leq \bar{L}x^{-1/\xi }$ for
all $x\geq x_{0}$. By Lemma 3.1.2 of \cite{Reiss89}, for all $u>0$, $\mathbb{%
P}(U_{k:n}\leq \frac{k}{n+1}u)\leq L(eu)^{k}$. Thus, $\mathbb{P}(Y_{n}\geq
y)\leq L(\bar{L}y^{-1/\xi }e)^{k}\leq Ly^{-1/\xi }$, where the last
inequality holds for all $y\geq (\bar{L}e)^{\xi }$, and the result follows.

(b) Clearly, $\mathbb{E}[\tilde{\tau}_{n}^{-\alpha }]\leq (n/k)^{-\alpha \xi
}\mathbb{E}[(\Gamma _{k}/k)^{\alpha \xi }]$. For $0\leq \alpha \xi \leq 1$, $%
\mathbb{E}[(\Gamma _{k}/k)^{\alpha \xi }]\leq \mathbb{E}[\Gamma
_{k}/k]^{1/(\alpha \xi )}=1$ while for $\alpha \xi >1$, $\mathbb{E}[(\Gamma
_{k}/k)^{\alpha \xi }]=\mathbb{E}[(k^{-1}\sum_{i=1}^{k}E_{i})^{\alpha \xi
}]\leq \mathbb{E}[k^{-1}\sum_{i=1}^{k}E_{i}^{\alpha \xi }]\leq C$ by two
applications of Jensen's inequality.
\end{proof}

\bigskip

\textbf{Proof of Theorem 1.}

We can assume $k\leq n^{\frac{2\delta }{1+2\delta }}$ in the following,
since otherwise, there is nothing to prove. Let $\tilde{T}_{n}=\max
(T_{n},x_{0})$. Lemma 3.1.1 in \cite{Reiss89} implies that under Condition
1, $\mathbb{P}(\tilde{T}_{n}\neq T_{n})\leq Ck/n$. Write $H_{n}(s)=\mathbb{P}%
(n^{-\gamma }S_{n}\leq s)$.

Assume first $1/3<\xi \leq 1/2$. We have 
\begin{equation*}
H_{n}(s)=\mathbb{E}\left[ \mathbb{P}\left( \sum_{i=1}^{n-k}\frac{X_{i:n}-\mu
(T_{n})}{(n-k)^{1/2}\sigma (T_{n})}\leq \frac{s/a_{n}-%
\sum_{i=k+1}^{n}X_{i:n}-(n-k)\mu (T_{n})}{(n-k)^{1/2}\sigma (T_{n})}%
|X_{n}^{e}\right) \right] .
\end{equation*}%
Note that conditional on $X_{n}^{e}$, the distribution of $%
\sum_{i=1}^{n-k}X_{i:n}$ is the same as that of the sum of i.i.d. draws from
the truncated distribution $\tilde{F}_{T_{n}}$ with mean $\mu (T_{n})$ and
variance $\sigma (T_{n})$. The Berry-Esseen bound hence implies 
\begin{equation*}
\sup_{z}\left\vert \mathbb{E}\left[ \mathbf{1}[\sum_{i=1}^{n-k}\frac{%
X_{i:n}-\mu (T_{n})}{(n-k)^{1/2}\sigma (T_{n})}\leq z]|X_{n}^{e}\right]
-\Phi (z)\right\vert \leq C(n-k)^{-1/2}\frac{\int |x|^{3}d\tilde{F}%
_{T_{n}}(x)}{\sigma ^{3}(T_{n})}
\end{equation*}%
where $\Phi (z)=P(Z\leq z)$. Replacing $T_{n}$ by $\tilde{T}_{n}$, by Lemma %
\ref{lm:con1} (e), $\int |x|^{3}d\tilde{F}_{\tilde{T}_{n}}(x)\leq C(\tilde{T}%
_{n})^{3-1/\xi }$ and $\sigma ^{3}(\tilde{T}_{n})\geq \sigma ^{3}(x_{0})$
a.s. From Lemma \ref{lm:expecT} (a), $\mathbb{E}[\tilde{T}_{n}^{3-1/\xi
}]\leq C(n/k)^{3\xi -1}$, so that%
\begin{equation*}
\sup_{s}\left\vert H_{n}(s)-\mathbb{E}\Phi \left( \frac{s/a_{n}-%
\sum_{i=k+1}^{n}X_{i:n}-(n-k)\mu (\tilde{T}_{n})}{(n-k)^{1/2}\sigma (\tilde{T%
}_{n})}\right) \right\vert \leq C(n^{-1/2}(n/k)^{3\xi -1}+k/n).
\end{equation*}%
From (\ref{Reiss555}), with $\tau _{n}=\omega (n/\Gamma _{k})^{\xi }$, $%
\sup_{s}\left\vert H_{n}(s)-H_{n}^{1}(s)\right\vert \leq
C(n^{-1/2}(n/k)^{3\xi -1}+(k/n)^{\delta }k^{1/2}+k/n)$, where 
\begin{equation*}
H_{n}^{1}(s)=\mathbb{E}\Phi \left( \frac{s/a_{n}-\omega n^{\xi
}\sum_{i=1}^{k}\Gamma _{i}^{-\xi }-(n-k)\mu (\tau _{n})}{(n-k)^{1/2}\sigma
(\tau _{n})}\right) .
\end{equation*}%
Let $\tilde{\tau}_{n}=\max (\tau _{n},x_{0})$ and note that by (\ref%
{Reiss555}), $P(\tilde{\tau}_{n}\neq \tau _{n})\leq \mathbb{P}(\tilde{T}%
_{n}\neq T_{n})+C((k/n)^{\delta }k^{1/2}+k/n)$.

Now focus on the claim in part (a). By Lemma \ref{lm:con1} (a) and (b), $%
|\mu (\tilde{\tau}_{n})+\frac{\omega ^{1/\xi }}{1-\xi }\frac{\tilde{\tau}%
_{n}^{1-1/\xi }}{1-(\tilde{\tau}_{n}/\omega )^{-1/\xi }}|\leq C\tilde{\tau}%
_{n}^{1-(1+\delta )/\xi }$ and $|\sigma ^{2}(\tilde{\tau}_{n})-\sigma
_{0}^{2}+\frac{\omega ^{1/\xi }}{1-2\xi }\tilde{\tau}_{n}^{2-1/\xi }|\leq
C\max (\tilde{\tau}_{n}^{2-(1+\delta )/\xi },\tilde{\tau}_{n}^{2-2/\xi },%
\tilde{\tau}_{n}^{-1/\xi })$ a.s. Thus, exploiting that $\phi (z)=d\Phi
(z)/dz$ and $|z|\phi (z)$ are uniformly bounded, and $0<\sigma
^{2}(x_{0})\leq \sigma ^{2}(\tilde{\tau}_{n})\leq \sigma _{0}^{2}$ a.s.,
exact first order Taylor expansions and Lemma \ref{lm:expecT} (b) yield 
\begin{multline*}
\sup_{s}\left\vert H_{n}^{1}(s)-\mathbb{E}\Phi \left( \frac{s-n^{\xi
-1/2}\omega \sum_{i=1}^{k}\Gamma _{i}^{-\xi }-n^{\xi -1/2}\frac{\omega }{%
1-\xi }\Gamma _{k}^{1-\xi }\psi _{n}}{\left( \sigma _{0}^{2}-\frac{\omega
^{2}}{1-2\xi }(\Gamma _{k}/n)^{1-2\xi }\right) _{+}^{1/2}}\right) \right\vert
\\
\leq C((k/n)^{\delta }k^{1/2}+k/n+n^{1/2}(n/k)^{\xi -1-\delta }+(n/k)^{2\xi
-(1+\delta )}+(n/k)^{2\xi -2})
\end{multline*}%
where $\psi _{n}=1+\frac{\Gamma _{k}/n-k/n}{1-\Gamma _{k}/n}$. Let $\tilde{%
\Gamma}_{k}=n(\tilde{\tau}_{n}/\omega )^{-1/\xi }$, so that $P(\tilde{\Gamma}%
_{k}\neq \Gamma _{k})=P(\tilde{\tau}_{n}\neq \tau _{n})$, and we can replace
any $\Gamma _{k}$ by $\tilde{\Gamma}_{k}$ in the last expression without
changing the form of the right hand side. Note that $1-\tilde{\Gamma}%
_{k}/n\geq 1-(x_{0}/\omega )^{-1/\xi }>0$ and $\sigma _{0}^{2}-\frac{\omega
^{2}}{1-2\xi }(\tilde{\Gamma}_{k}/n)^{1-2\xi }\geq \sigma ^{2}(x_{0})$ a.s.
Thus, by another exact Taylor expansion and $\mathbb{E}[\Gamma _{k}^{1-\xi
}|\Gamma _{k}/n-k/n|]^{2}\leq \mathbb{E}[\Gamma _{k}^{2-2\xi }]\mathbb{E}%
[(\Gamma _{k}/n-k/n)^{2}]\leq Ck^{3-2\xi }/n^{2}$, we can replace $\psi _{n}$
by $1$ at the cost of another error term of the form $C(n/k)^{-3/2+\xi }$.
The result in part (a) now follows after eliminating dominated terms, and
the proof of part (b) for $1/3<\xi <1/2$ follows from the same steps.

So consider $\xi =1/2$. Let $A_{n}$ be the event $(2k)^{-\xi }\leq \Gamma
_{k}^{-\xi }\leq (k/2)^{-\xi }$. By Chebychev's inequality, $\mathbb{P}%
(A_{n})=\mathbb{P}(1/2\leq k^{-1}\sum_{i=1}^{k}E_{i}\leq 2)\leq C/k$.
Conditional on $A_{n}$, and recalling that $k\leq n^{\frac{2\delta }{%
1+2\delta }}$, $C^{-1}\leq \sigma ^{2}(\tilde{\tau}_{n})/\log (n)\leq C$, $%
|\sigma ^{2}(\tilde{\tau}_{n})-\sigma ^{2}(\omega u_{n})-\frac{\omega
^{1/\xi }}{\xi }\int_{\omega u_{n}}^{\tau _{n}}y^{1-1/\xi }dy|\leq
C((n/k)^{2\xi -1-\delta }+(k/n)\log (n))$ and $|\mu (\tilde{\tau}_{n})+\frac{%
\omega ^{1/\xi }}{1-\xi }\frac{\tilde{\tau}_{n}^{1-1/\xi }}{1-(\tilde{\tau}%
_{n}/\omega )^{-1/\xi }}|\leq C(n/k)^{\xi -1-\delta }$ a.s.~ by Lemma \ref%
{lm:con1} (a) and (d). Exact first order Taylor expansions of $H_{n}^{1}(s)$
thus yield%
\begin{multline*}
\sup_{s}\left\vert H_{n}^{1}(s)-\mathbb{E}\Phi \left( \frac{s(\log
n)^{1/2}-\omega n^{1/2}\sum_{i=1}^{k}\Gamma _{i}^{-1/2}-n^{1/2}\frac{\omega 
}{1-\xi }\Gamma _{k}^{1/2}\psi _{n}}{\left( \sigma ^{2}(\omega
u_{n})+2\omega ^{2}\int_{u_{n}}^{(n/\Gamma _{k})^{1/2}}y^{-1}dy\right)
_{+}^{1/2}}\right) \right\vert \\
\leq C(k^{-1/2}+(k/n)^{\delta }k^{1/2}+k/n+n^{1/2}(n/k)^{-1/2-\delta
}+(k/n)^{-\delta })
\end{multline*}%
and replacing $\psi _{n}$ by unity induces an additional error term of the
form $C(n/k)^{-1}$ by the same arguments as employed above (and recalling
that $\mathbb{P}(A_{n})\leq C/k$).

We are left to prove the claim for $1/2<\xi <1$. Note that the distribution
of $\sum_{i=1}^{n-k}X_{i:n}$ conditional on $X_{n}^{e}$ only depends on $%
X_{n}^{e}$ through $T_{n}$. Let $\Phi _{n,t}$ be the conditional
distribution function of $\sum_{i=1}^{n-k}\frac{X_{i:n}-\mu (T_{n})}{%
(n-k)^{1/2}\sigma (T_{n})}$ given $T_{n}=t$. For future reference, note that
by Theorem 1.1 in \cite{Goldstein10}, $||\Phi _{n,t}-\Phi ||_{1}=\int |\Phi
(z)-\Phi _{n,t}(z)|dz\leq (n-k)^{-1/2}\int |y|^{3}dF_{t}(y)/\sigma (t)^{3}$,
so that by Lemma \ref{lm:con1} (c) and (e), $||\Phi _{n,t}-\Phi ||_{1}\leq
Cn^{-1/2}t^{1/(2\xi )}$ for $t\geq x_{0}$. We have%
\begin{equation*}
H_{n}(s)=\mathbb{E}\Phi _{n,T_{n}}\left( \frac{n^{\xi
}s-\sum_{i=k+1}^{n}X_{i:n}-(n-k)\mu (T_{n})}{(n-k)^{1/2}\sigma (T_{n})}%
\right)
\end{equation*}%
so that by (\ref{Reiss555}), $\sup_{s}\left\vert
H_{n}(s)-H_{n}^{2}(s)\right\vert \leq C((k/n)^{\delta }k^{1/2}+k/n)$, where 
\begin{equation*}
H_{n}^{2}(s)=\mathbb{E}\Phi _{n,\tau _{n}}\left( \frac{n^{\xi }s-\omega
n^{\xi }\sum_{i=1}^{k}\Gamma _{i}^{-\xi }-(n-k)\mu (\tau _{n})}{%
(n-k)^{1/2}\sigma (\tau _{n})}\right) .
\end{equation*}

Let $U$ be a uniform random variable on the unit interval, independent of $%
(\Gamma _{i})_{i=1}^{\infty }$, and let $\Phi _{n,t}^{-1}$ be the quantile
function of $\Phi _{n,t}$. Then%
\begin{equation*}
H_{n}^{2}(s)=\mathbb{P}\left( n^{-\xi }(n-k)^{1/2}\sigma (\tau _{n})\Phi
_{n,k,\tau _{n}}^{-1}(U)+n^{-\xi }(n-k)\mu (\tau _{n})+\omega
\sum_{i=1}^{k}\Gamma _{i}^{-\xi }\leq s\right) .
\end{equation*}%
Since $\Gamma _{1}/\Gamma _{2},\Gamma _{2}/\Gamma _{3},\ldots ,\Gamma
_{k-1}/\Gamma _{k},\Gamma _{k}$ are independent (cf. Corollary 1.6.11 of 
\cite{Reiss89}), the distribution of $(\Gamma _{1}/\Gamma _{2})^{-\xi }$
conditional on $\Gamma _{2},\Gamma _{3},\ldots ,\Gamma _{k}$ is the same as
that conditional on $\Gamma _{2}$, which by a direct calculation is found to
be Pareto with parameter $1/\xi $. Thus, with $G(z)=\mathbf{1}%
[z>1](1-z^{-1/\xi })$, 
\begin{equation*}
H_{n}^{2}(s)=\mathbb{E}G\left( \frac{s-n^{-\xi }(n-k)^{1/2}\sigma (\tau
_{n})\Phi _{n,\tau _{n}}^{-1}(U)-n^{-\xi }(n-k)\mu (\tau _{n})-\omega
\sum_{i=2}^{k}\Gamma _{i}^{-\xi }}{\omega \Gamma _{2}^{-\xi }}\right) .
\end{equation*}%
Note that for arbitrary $a\geq 0$ and $y\in 
%TCIMACRO{\U{211d} }%
%BeginExpansion
\mathbb{R}
%EndExpansion
$, with $g(z)=dG(z)/dz$ 
\begin{eqnarray*}
|\mathbb{E}G(y+a\Phi _{n,t}^{-1}(U))-\mathbb{E}G(y+aZ)| &=&|\int
G(y+az)d(\Phi _{n,t}(z)-\Phi (z))| \\
&=&a|\int (\Phi (z)-\Phi _{n,t}(z))g(y+az)dz| \\
&\leq &a\sup_{y}|g(y)|\cdot ||\Phi _{n,t}-\Phi ||_{1}
\end{eqnarray*}%
where the second equality stems from Riemann-Stieltjes integration by parts.
Conditional on the event $A_{n}$ as defined above, $||\Phi _{n,\tilde{\tau}%
_{n}}-\Phi ||_{1}\leq Ck^{-1/2}$, $C^{-1}(n/k)^{2\xi -1}\leq \sigma ^{2}(%
\tilde{\tau}_{n})\leq C(n/k)^{2\xi -1}$, $|\sigma ^{2}(\tilde{\tau}%
_{n})-\sigma ^{2}(\omega u_{n})-\frac{\omega ^{1/\xi }}{\xi }\int_{\omega
u_{n}}^{\tilde{\tau}_{n}}y^{1-1/\xi }dy|\leq C((n/k)^{2\xi -1-\delta
}+(n/k)^{2\xi -2})$ and $|\mu (\tilde{\tau}_{n})+\frac{\omega ^{1/\xi }}{%
1-\xi }\frac{\tilde{\tau}_{n}^{1-1/\xi }}{1-(\tilde{\tau}_{n}/\omega
)^{-1/\xi }}|\leq C(n/k)^{\xi -1-\delta }$ a.s.~by Lemma \ref{lm:con1} (a)
and (c). Thus, by exact first order Taylor expansions and exploiting that $%
g(z)$ is uniformly bounded and $\mathbb{E}[|Z|]$, $\mathbb{E}[\Gamma
_{2}^{\xi }]<C$, 
\begin{multline*}
\sup_{s}|H_{n}^{2}(s)-H_{n}^{3}(s)|\leq \\
C(k^{-1}+(k/n)^{\delta }k^{1/2}+k/n+k^{-\xi }+n^{1-\xi }(n/k)^{\xi -1-\delta
}+n^{1/2-\xi }((n/k)^{\xi -1/2-\delta }+(n/k)^{\xi -3/2}))
\end{multline*}%
where%
\begin{equation*}
H_{n}^{3}(s)=\mathbb{E}G\left( \frac{s-\omega \sum_{i=2}^{k}\Gamma
_{i}^{-\xi }-n^{1/2-\xi }\left( \sigma ^{2}(\omega u_{n})+\frac{\omega
^{1/\xi }}{\xi }\int_{\omega u_{n}}^{\tau _{n}}y^{1-1/\xi }dy\right)
_{+}^{1/2}Z+\frac{\omega }{1-\xi }\Gamma _{k}^{1-\xi }\Psi _{n}}{\omega
\Gamma _{2}^{-\xi }}\right) .
\end{equation*}%
As before, we can replace $\Psi _{n}$ by unity at the cost of another error
term of the form $Ck^{3/2-\xi }/n$, and the result follows after eliminating
dominating terms.

\bigskip

\bibliographystyle{econometrica}
\bibliography{diss}

\end{document}